\documentclass[12pt]{amsart}
\usepackage{amsmath}
\usepackage{amsfonts}
\usepackage{amssymb,epstopdf,float}
\usepackage{pifont}
\usepackage{cite}
\usepackage{appendix}
\usepackage[dvips]{color}
\usepackage{graphicx}
\usepackage{subfigure}

\newtheorem{theorem}{Theorem}[section]
\newtheorem{thm}[theorem]{Theorem}
\newtheorem{prop}{Proposition}[section]
\newtheorem{cor}[prop]{Corollary}
\newtheorem{lemma}[prop]{Lemma}
\newtheorem{remark}[prop]{Remark}
\newtheorem{defi}[prop]{Definition}

\newtheorem{example}[prop]{Example}
\numberwithin{equation}{section}

\def\R{\mathbb R}

\def\mathscr{\mathcal }

\def\diag{\text{diag}}
\def\mi{\mathbf {i}}

\newcommand{\ba}{{\mathbf{a}}}

\newcommand{\bd}{{\mathbf{d}}}
\newcommand{\be}{{\mathbf{e}}}

\newcommand{\bh}{{\mathbf{h}}}

\newcommand{\bD}{{\mathcal{D}}}

\newcommand{\SD}{{\mathcal D}}

\begin{document}

\title{Dimension drop of connected part   Of slicing self-affine  Sponges}

\author{Yan-fang Zhang}
\address{School of Science, Huzhou University, Huzhou, 313000, China;}
\email{03002@zjhu.edu.cn}
\author{Yan-li Xu$\dag$}
\address{Department of Mathematics and Statistics, Central China Normal University, Wuhan,430079, China.}
\email{xu\_yl@mails.ccnu.edu.cn}

\date{\today}
\thanks{The work is supported by the start-up research fund from Huzhou University. No. RK21089.}

 \begin{abstract}
   The connected part of a metric space $E$ is defined to be the union of non-trivial connected components of $E$. We proved that for a class of self-affine sets called slicing self-affine sponges, the connected part of $E$ either coincides with $E$, or is essentially contained in the attractor of  a proper sub-IFS of an iteration of the original IFS.
This generalize an early result of  Huang and Rao  [L. Y. Huang, H. Rao. \emph{A dimension drop phenomenon of fractal cubes}, J. Math. Anal. Appl. \textbf{497} (2021), no. 2] on  a class of self-similar sets called fractal cubes. Moreover, we show that  the result is no longer valid if the slicing property is removed.
    Consequently, for a Bara\'nski carpet $E$,  the Hausdorff dimension and the box dimension of the connected part of $E$ are strictly less than the Hausdorff dimension and the box dimension of $E$, respectively. For slicing self-affine sponges in $\mathbb R^d$ with $d\geq 3$,
    whether the attractor of a sub-IFS has strictly smaller dimensions is an open problem.
 \end{abstract}

\thanks{$\dag$ The corresponding author}

\maketitle

\section{\textbf{Introduction}}
Usually a fractal set contains both trivial connected components and non-trivial connected components,
where  a connected component is called trivial if it is a singleton.
 Let $X$ be a metric space. Let $X_c$ be the union of non-trivial connected components and we call it
 the \emph{connected part} of $X$.
 Huang and Rao \cite{HR20} introduced a notion of connectedness index of  $X$  as
\begin{equation}\label{eq:Cindex}
\text{ind}_H X=\dim_H X_c,
\end{equation}
where $\dim_H$ denotes the Hausdorff dimension. In the following, we will call $\text{ind}_H X$
 the \emph{Hausdorff connectedness index} of $X$. It is shown that (\cite{HR20}) if $X$ is a fractal cube,
 that is, $X\subset \R^d$ is the non-empty compact set satisfying the set equation
 $
 X=\bigcup_{\bd\in \SD}(X+\bd)/n,
 $
 where $n\geq 2$ is an integer and $\SD\subset \{0,1,\dots, n-1\}^d$,
 then either $X_c=X$ or $\text{ind}_H X=\dim_H X_c<\dim_H X$.

 This result has several interesting consequences. Let $X$ be a fractal cube possessing trivial points.
 First, the topological Hausdorff dimension of $X$ is less than $\dim_H X$,  since
 the former one is no larger than $\text{ind}_H$, see \cite{HR20} and \cite{YZ20}. (The topological Hausdorff dimension is a new dimension introduced by \cite{BBE15}.)
 Secondly,  the gap sequence of $X$ is comparable to $(k^{-\beta})_{k\geq 1}$
 where $\beta=\dim_H X$, see \cite{HZ20}.
 Thirdly, $\text{ind}_H X$ provides a new Lipschitz invariant.

The aim of the present paper is to generalize the results of \cite{HR20} to a larger class of fractals
called slicing self-affine sponges.

\begin{defi}[Diagonal IFS (\cite{Das16})]\label{IFS}\emph{
Fix $d\geq1$. For each $i\in \{1,\ldots,d\}$, let $A_i=\{0,1,\dots, n_i-1\}$ with $n_i\geq 2$, and let $\Phi_i=(\phi_{a,i})_{a\in A_i}$ be a collection of contracting similarities of $[0,1]$, called
the \emph{base IFS in coordinate} $i$.
Let $A =\prod _{i=1}^d A_i$, and for each $\textbf{a}=(a_1,\ldots,a_d)\in A,$ consider the contracting affine maps $\phi_\textbf{a}$:$[0,1]^d \rightarrow [0,1]^d$ defined by the formula
$${\phi_\textbf{a}}\left({{x_1},\ldots,{x_d}}\right)=\left({{\phi_{\textbf{a},1}}\left({{x_1}}\right),\ldots,{\phi_{\textbf{a},d}}\left({{x_d}}\right)}\right),$$
where $\phi_{\textbf{a},i}$ is shorthand for $\phi_{a_i,i}$ in the formula above.
Then we can get
$${\phi_\textbf{a}}\left({{{[0,1]}^d}}\right)=\mathop\Pi\limits_{i=1}^d{\phi_{\textbf{a},i}}\left({\left[{0,1}\right]}\right)\subset{\left[{0,1}\right]^d}.$$
Given $\SD\subset A$, we call the collection $\Phi={\left({{\phi_\ba}}\right)_{\ba\in \SD}}$ a \emph{diagonal IFS}, and we call its invariant set $\Lambda_\Phi$ a \emph{diagonal self-affine sponge}.
}
\end{defi}

A diagonal self-affine IFS  $\Phi$  is called  a \emph{slicing self-affine IFS},
  if for each $1\leq i \leq d$,
  $$[0,1)=\phi_{0,i}[0,1)\cup \cdots \cup \phi_{n_i-1,i}[0,1)$$
  is a partition of $[0,1)$ from left to right;
  in this case, we call  $\Lambda_\Phi$   a \emph{slicing self-affine sponge}.
   In the two dimensional case,   $\Lambda_\Phi$ is called
  a \emph{Bara\'nski carpet} (\cite{Das16}).

Let $\Lambda_\Phi$ be a slicing self-affine sponge.
If for each $i$, the maps in $\Phi_i$ have equal contraction ratio $1/n_i$, then $\Lambda_\Phi$
 is called  a \emph{Sierpi\'nski self-affine sponge}, see Kenyon and Peres \cite{KP96measure} and Olsen \cite{Olsen07}.
 There is a simple way to define a Sierpi\'nski self-affine sponge.
  Let $M =\text{diag}(n_1,\dots, n_d)$ and let $\SD\subset \prod_{i=1}^d \{0,1,\dots, n_i-1\}$.
Then $M$ and $\mathcal {D}$ determine an IFS
\begin{equation}\label{eq:sier}
 {\left\{ {{\lambda _\bd}\left( z \right) = {M^{ - 1}}\left( {z + \bd} \right)} \right\}_{\bd \in \mathcal {D}}}.
\end{equation}
Then its invariant set  $K = K\left( {M,\mathcal {D}} \right)$   is a Sierpi\'nski self-affine sponge.
 If $M$ and $\SD$ are obtained from a slicing IFS $\Phi$, then  we call it the \emph{associated Sierpi\'nski self-affine sponge} of   $\Phi$.

\begin{remark}
\emph{ There are a lot of works on dimensions, multifractal anlaysis
and other topics of  diagonal self-affine sponges,  see for instance,
McMullen \cite{Mc84}, Bedford \cite{Bed84}, Lalley and Gatzouras   \cite{Lalley92}, King \cite{King95},
Kenyon and Peres \cite{KP96sofic, KP96measure},
Feng and Wang \cite{FengWang05},
Bara\'nski \cite{Baranski07}, Olsen \cite{Olsen07}, Barral and Mensi  \cite{Bar07},
Jordan and Rams \cite{JR11}, Mackay \cite{MM11}, Das and Simmons \cite{Das16}, Fraser \cite{Fraser2017},
Miao, Xi and Xiong \cite{Miao2017}, Rao, Yang and Zhang \cite{Rao2019}.
}
\end{remark}

Our main results are as follows.
First, we prove that a slicing self-affine sponge is always homeomorphic to its associated Sierpi\'nski
self-affine sponge. We use $\partial E$ to denote the  boundary of a set $E$.

\begin{thm}\label{thm:homeo} 
Let $\Lambda_\Phi$ be a $d$-dimensional slicing self-affine sponge generated by $\Phi=(\phi_{\ba})_{\ba\in \SD}$,
and  $K(M,\SD)$ be the associated Sierpi\'nski self-affine sponge of $\Phi$.
Then
there exists a map $F: \Lambda_\Phi\to K(M,\SD)$   which is bi-Holder continuous, and
$F(\Lambda_\Phi\cap \partial [0,1]^d)\subset \partial [0,1]^d$.
\end{thm}

A slicing  self-affine sponge is said to be \emph{degenerated}, if it is contained in a $(d-1)$-face
of $[0,1]^d$. (For definition of face of a convex body, see Section 3.)
Next, we show that

\begin{thm}\label{thm:inner} Let $\Lambda_\Phi$ be a slicing self-affine sponge.
If $\Lambda_\Phi$ is non-degenerated and  possesses trivial points, then $\Lambda_\Phi$
has a trivial point in $(0,1)^d$.
\end{thm}

\begin{remark}\emph{Thanks to Theorem \ref{thm:homeo}, we only need to proof Theorem \ref{thm:inner}
for Sierpi\'nski self-affine sponge, which generalizes the main result of \cite{HR20}. The key point in our
proof is that we build up Lemma 3.3, which allows us to generalize and simplify the approach of \cite{HR20}.
Besides, our definition of degeneration of self-affine sponge is more suitable for the purpose than that in \cite{HR20}.
}
\end{remark}

For $\omega=\omega_1\dots \omega_k\in \SD^k$, we denote $\phi_{\omega}=\phi_{\omega_1}\circ\cdots \circ \phi_{\omega_k}$.
Thanks to Theorem \ref{thm:inner}, we can show that

\begin{thm}\label{thm:subIFS}  Let $\Lambda_\Phi$ be a slicing self-affine sponge with digit set $\SD$.
If $\Lambda_\Phi$ possesses trivial points, then there exists a proper sub-IFS of an iteration of $\Phi$ with invariant set $\Lambda'$ such that
 the connected part of $\Lambda_\Phi$ is contained in $\bigcup_{\omega\in \cup_{k\geq 0} \SD^k}\phi_\omega(\Lambda')$
\end{thm}

The example below shows that the `slicing property' in
 Theorem \ref{thm:inner} cannot be removed.

\begin{example}\label{exam:counter}
\emph{In this example, we identify $\R^2$ with the complex plane $\mathbb C$.
Let $E$ be the self-similar set generated by the IFS  $\{f_i\}_{i=1}^7$, where
$$
f_1(z)= \frac{2z}{3},
f_2(z)= \frac{z}{3} + \frac{{2\mi}}{3},
f_3(z)= \frac{z}{3} + \frac{{1 + 2\mi}}{3},
f_4(z)= \frac{z}{6} + \frac{{4 + 5\mi}}{6},
$$
$$
f_5(z)= \frac{z}{6} + \frac{{5+ 5\mi}}{6},
f_6(z)= \frac{z}{3} + \frac{{2 + \mi}}{3},
f_7(z)= \frac{z}{{12}} + \frac{{22 + 17\mi}}{{24}}.
$$
We will show that the trivial points of $E$ are all located on the line segment $\{1\}\times [1/2,1]$
(Theorem \ref{thm:tri}).
}
\begin{figure}
\subfigure[The digit set of $E$.]{\includegraphics[width=4.5 cm]{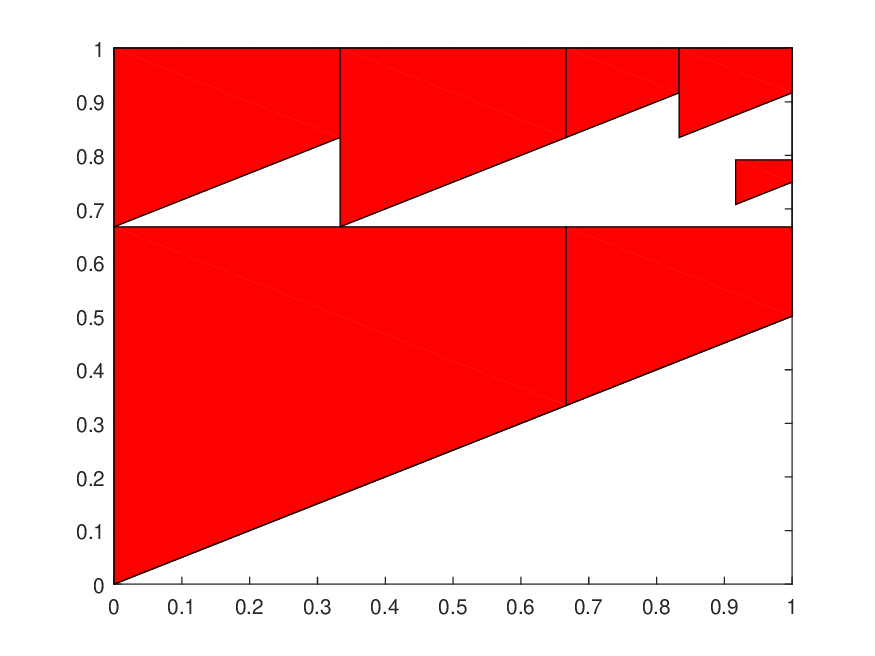}}\quad\quad
\subfigure[The slicing sponge $E$.]{\includegraphics[width=4.5 cm]{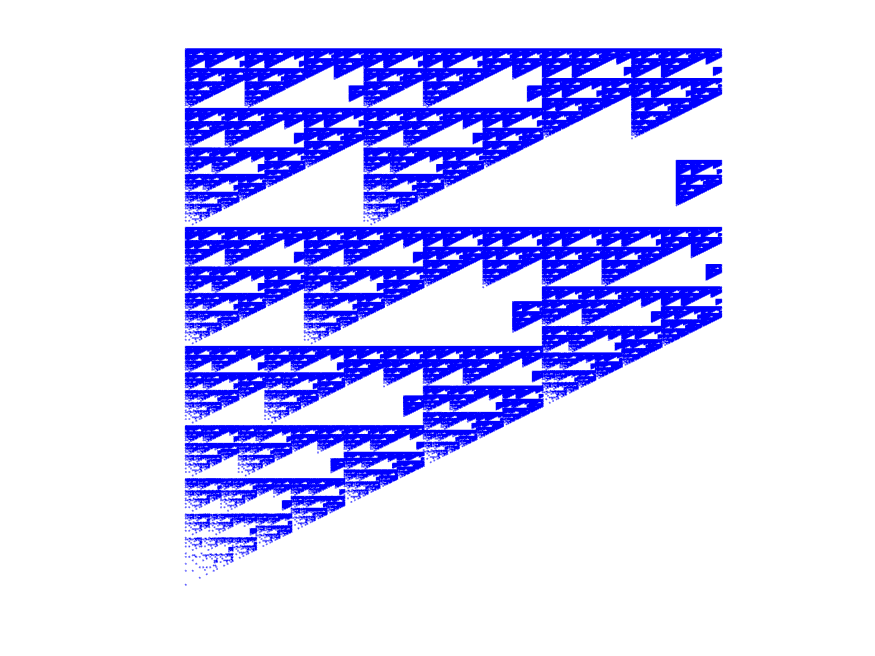}}
\caption{}
\label{T}
\end{figure}

\end{example}

Let $X$ be a metric space. Similar to the Hausdorff connectedness index, we define the \emph{box connectedness index}
of $X$ to be
\begin{equation}\label{eq:Bindex}
\text{ind}_B X=\dim_B X_c,
\end{equation}
where $X_c$ is the  connected  part of $X$.

Let $\Phi_i$, $i=1,\dots, d$ be the base IFSes in definition of slicing IFS. We denote by $\Phi(\SD)$ the slicing IFS determined by $\SD$.
There is an important open question in dynamical system that if $\SD'$ is a proper subset of $\SD$, is it true
that
\begin{equation}\label{eq:dropp}
\dim_H \Lambda_{\Phi(\SD')} < \dim_H \Lambda_{\Phi(\SD)}\text{ and }
\dim_B \Lambda_{\Phi(\SD')} < \dim_B \Lambda_{\Phi(\SD)}?
\end{equation}
(See for example, Kanemaki \cite{Kanemaki}.)

Till now we can only show that the above inequalities hold for Bara\'nski carpets.
In the box dimension case, Bara\'nski \cite{Baranski07} gave an implicit formula for box dimension,
which confirms the second inequality of \eqref{eq:dropp} easily.
As for the Hausdorff dimension,   Bara\'nski showed that a certain variation  principle
holds where only the Bernoulli measures are involved. Hence, it is not hard to show that the
maximum cannot occur at a boundary point of the (finite dimensional) parameter space (It is proved in a manuscript of Dejun Feng \cite{Feng}). Hence, we have

\begin{cor}\label{thm:dimension decrease} If a Bara\'nski carpet $\Lambda_\Phi$ has a trivial point, then $$\text{ind}_H(\Lambda_\Phi)<\dim_H(\Lambda_\Phi) \text{ and }
\text{ind}_B(\Lambda_\Phi)<\dim_B(\Lambda_\Phi).$$
\end{cor}

 But when $d\geq 3$, Das and Simmons \cite{Das16} proved
that the Hausdorff dimension can be obtained by a variation  principle which involves a larger class of measures called pseudo Bernoulli measures. So it is not clear whether  the
first inequality of  \eqref{eq:dropp} holds.  Also, till now,  we do not have a formula
of box dimension of slicing self-affine sponge   with $d\geq 3$.

\begin{example}\label{examE}
\emph{Let $E$ and $E'$ be two Bedford-McMullen carpets indicated by Figure \ref{E}.
 They have
 the same Hausdorff dimension and also have the same box dimension. We will  calculate the connectedness indices
 of $E$ and $E'$ in Section 5, see Table 1.
 In fact, the connected part of $E$ is essentially a \emph{sofic self-affine carpet} studied by Kenyon
 and Peres \cite{KP96sofic}.
 It is seen that  $\text{ind}_H(E)\neq \text{ind}_H(E')$ and $\text{ind}_B(E)\neq \text{ind}_B(E')$, so $E$ and $E'$ are not Lipschitz equivalent.}

 \begin{table}[H]
\centering
\caption{Dimensions and indices of $E$ and $E'$}
\label{tab1}
\begin{tabular}{|l|c|c|c|c|}
\hline
    & $dim_H$ & $ dim_B$ &$ \text{ind}_H $ &   $ \text{ind}_B$ \\ \hline
  $E$ & ${\log _5}\left( {12 + \sqrt[3]{5}} \right)\approx1.627$ & $1 + {\log _8}4\approx1.667$ & $\approx1.61$ & $\frac{{\log \lambda  + \left( {\sigma  - 1} \right)\log 5}}{{\log 8}}\approx 1.662$ \\ \hline
  $E'$& ${\log _5}\left( {12 + \sqrt[3]{5}} \right)$&$ 1 + {\log _8}4$ & $\approx1.54$ & $\frac{{\log 19 + \left( {\sigma  - 1} \right)\log 5}}{{\log 8}}\approx1.640$\\ \hline
\end{tabular}
\noindent{In the above table, $\lambda=\frac{{22+\sqrt {312}}}{2}$ and   $\sigma=\log 8/\log 5$.}
\end{table}

\begin{figure}[H]
\subfigure[The digit set of $E$.]{\includegraphics[width=4.5 cm]{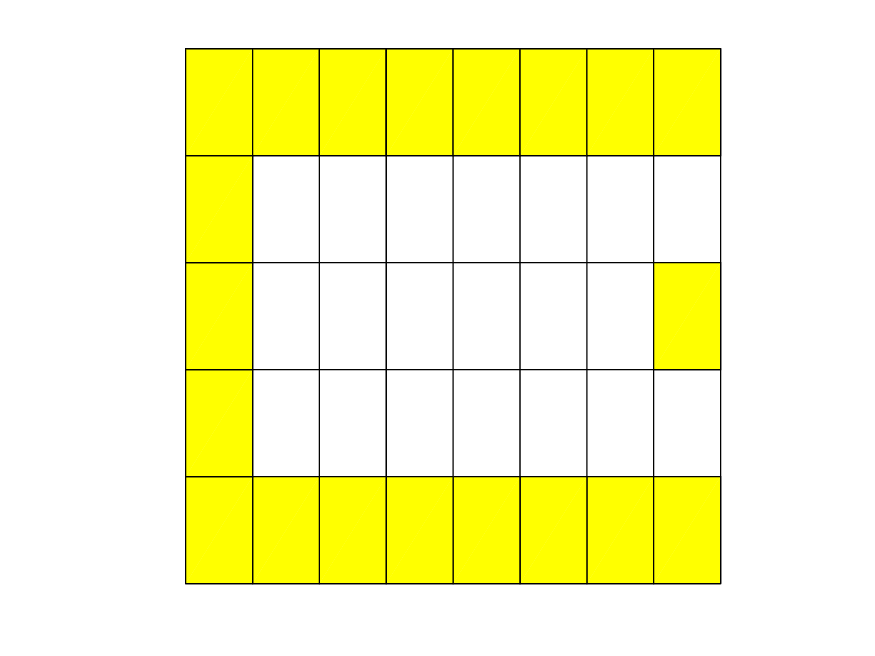}}\quad\quad
\subfigure[The digit set of $E'$.]{\includegraphics[width=4.5 cm]{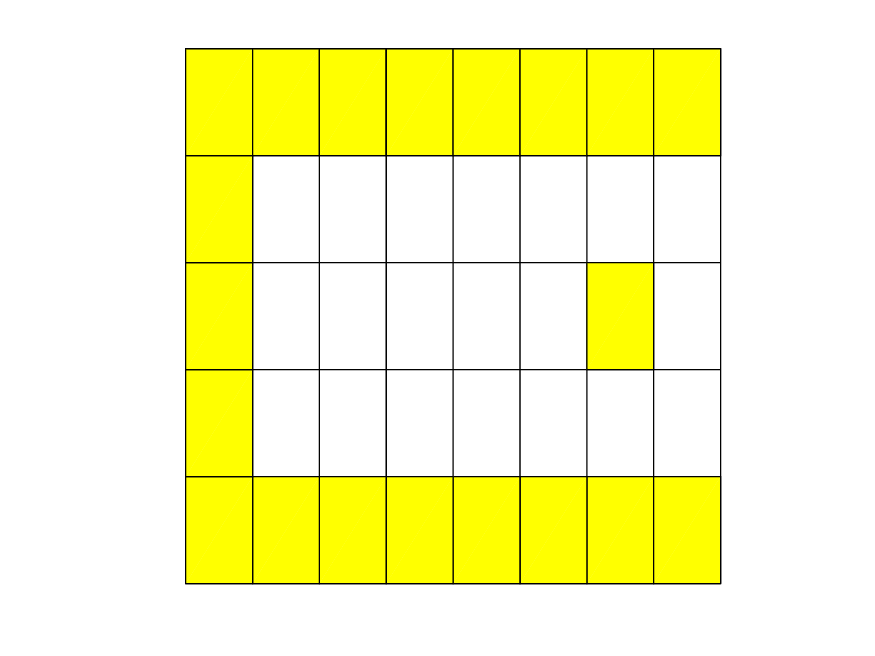}}
\caption{}
\label{E}
\end{figure}
\end{example}

This article is organized as follows. Theorem \ref{thm:homeo} is proved in section 2. In section 3, we recall some basic facts about faces of the polytope $[0,1]^d$. Theorem \ref{thm:inner} and Theorem \ref{thm:subIFS} are proved in section 4. Example \ref{examE} and Example \ref{exam:counter}  are discussed in Section 5 and Section 6, respectively.

\section{\textbf{Proof of Theorem \ref{thm:homeo}} }

First, let us give some notations about IFS. Let $A$ be a finite set.
An \emph{iterated function system (IFS)} on $\R^d$  is a family of contractions $\Phi=\{\phi_{\ba}\}_{\ba\in A}$ on ${\R}^{d}$, and the \textit{attractor} of the IFS is the unique nonempty compact set \textit{K} satisfying $K={\bigcup}_{\ba\in A}\phi_{\ba}{(K)}$, and it is called a \textit{self-similar set} if all the contractions are similitudes. (See \cite{Hut81}.)
Let $\pi_\Phi: A^\infty\to K$ be the coding map defined by
$$
\{\pi_\Phi((a_k)_{k\geq 1})\}=\bigcap_{k\geq 1} \phi_{a_1\dots a_k}(K).
$$
We say  $(a_k)_{k\geq 1}$ is a \emph{coding} of $x\in K$ if $x=\pi_\Phi((a_k)_{k\geq 1})$.

Let $\Phi$ be a slicing IFS with digit set $\SD$. Let $\Lambda:=\Lambda_\Phi$ be a slicing self-affine sponge.
Let $\pi: \SD^\infty\to \Lambda_\Phi$ be the corresponding coding map.
Let $K:=K(M,\SD)$ be the associated Sierpi\'nski sponge of $\Phi$, let $(\lambda_\bd)_{\bd\in \SD}$
be an generating IFS of $K$ defined by  \eqref{eq:sier}, and
 $\pi': \SD^\infty \to K(M,\SD)$ be the corresponding coding map.

 First, we build a simple lemma.
Let $n\geq 2$ and let
\begin{equation}\label{eq:contr}
\Psi:=(\psi_j=r_jx+b_j)_{j=0}^{n-1}
\end{equation}
 be a family of contractions on $\R$  with $0<r_j<1$,
such that $\psi_0([0,1))\cup\cdots \cup \psi_{n-1}([0,1))$ is a partition of $[0,1)$ from left to right.

For $\bh =(h_k)_{k\geq 1}\in {\left\{ {0,1, \ldots ,n - 1} \right\}^\infty }$, we define two functions
\begin{equation}\label{fun:f}
f(\bh)= \sum_{k\geq 1} \frac{{{h_k}}}{{{n^k}}},
\end{equation}
and
\begin{equation}\label{fun:g}
 g(\bh)= \sum_{k\geq 1}     r_ {h_1}    \cdots  r_ {h_{k - 1}}   {b_{{h_k}}}.
\end{equation}
Note that the first term of the right hand side of \eqref{fun:g} is $b_{h_1}$.

\begin{lemma}\label{lem:Holder} There exists $\alpha>0$ such that for any   $\bh,\bh'\in \{0,1,\dots, n-1\}^\infty$,
$$
1/2\cdot |g(\bh)-g(\bh')|^{1/\alpha} \leq |f(\bh)-f(\bh')|\leq 2|g(\bh)-g(\bh')|^\alpha.
$$
\end{lemma}

\begin{proof}  Denote $r^*=\max\{r_i;~ 0\leq i\leq n-1\}$ and  $r_*=\min\{r_i;~ 0\leq i\leq n-1\}$.

Write $\bh=(h_k)_{k\geq 1}$ and $\bh'=(h'_k)_{k\geq 1}$.
Let $k$ be the smallest integer such that $h_k\neq h'_k$. Without loss of generality, we assume
that $h_k>h'_k$.

If $h_k-h'_k\geq 2$, we set $k^*=k$, otherwise, that is, if $ h_k -  h'_k  = 1$, we set
$$
k^*=\max\left \{\ell\geq k+1;~ \left( {\begin{array}{*{20}{c}}{{h_{k+1}\dots h_{\ell-1}}}  \\
   {{h'_{k+1}\dots h'_{\ell-1}}}  \\
\end{array}} \right) = \left( {\begin{array}{*{20}{c}}
   {0}  \\
   n-1  \\
\end{array}} \right)^{\ell-k-1}\right \}.
$$
Then

\begin{equation}\label{eq:Holder-1}
n^{-k^*}\leq |g(\bh)-g(\bh')|\leq 2n^{-k^*+1}.
\end{equation}

Similarly, we have
\begin{equation}\label{eq:Holder-2}
(r_*)^{-k^*}\leq |f(\bh)-f(\bh')|\leq 2(r^*)^{-k^*+1}.
\end{equation}
Setting $\alpha=\min\{ -\log r^*/\log n, -\log n/\log r_*\}$,
we obtain the lemma.
\end{proof}

Write $\pi=(\pi_1,\dots, \pi_d)$ and $\pi'=(\pi'_1,\dots, \pi'_d)$.

\begin{lemma}\label{lem:coding}  Let  $\overrightarrow \ba=(\ba_k)_{k\geq 1}, \overrightarrow {\ba'}=(\ba'_k)_{k\geq 1}\in \SD^\infty$. Then
$\pi(\overrightarrow \ba)=\pi(\overrightarrow {\ba'})$ if and only if
$\pi'(\overrightarrow \ba)=\pi'(\overrightarrow {\ba'})$.
\end{lemma}

\begin{proof} Notice that both $\phi_{\bd}(z)$ and $\lambda_{\bd}(z)$ are variable separation functions.
Clearly $\pi(\overrightarrow \ba)=\pi(\overrightarrow {\ba'})$ if and only if
\begin{equation}\label{eq:separation}
\pi_j(\overrightarrow \ba)=\pi_j(\overrightarrow {\ba'}),  \quad \text{ for  } 1\leq j\leq d.
\end{equation}
By Lemma \ref{lem:Holder}, \eqref{eq:separation} holds if and only if
$$
\pi'_j(\overrightarrow \ba)=\pi'_j(\overrightarrow {\ba'}),  \quad \text{ for } 1\leq j\leq d,
$$
which is equivalent to $\pi'(\overrightarrow \ba)=\pi'(\overrightarrow {\ba'})$.
The lemma is proved.
\end{proof}

 \begin{proof}[Proof of Theorem \ref{thm:homeo}]
 According to Lemma \ref{lem:coding}, we define a coding preserving map
$F: \Lambda_\Phi\to K(M,\SD)$ by
\begin{equation}\label{eq:F}
F(z)=\pi'\circ \pi^{-1}(z).
\end{equation}
 Write $F(z)=(F_1(z),\dots, F_d(z))$.
 Then $F_j(z)=\pi'_j\circ\pi_j^{-1}(z)$ for $1\leq j\leq d$. By Lemma \ref{lem:Holder},
 all $F_j$ are bi-Holder continuous, so $F$ is also bi-Holder continuous.

 Suppose $z\in \Lambda_{\Phi}\cap \partial [0,1]^d$. Denote $(\ba_k)_{k\geq 1}=\pi^{-1}(z)$.
 Then there exists $j$ such that $\phi_{\ba_k,j}=\phi_{0,j}$
 for all $k\geq 1$, or $\phi_{\ba_k,j}=\phi_{n_k-1,j}$for all $k\geq 1$.
 It follows that $\ba_{k,j}=0$ for all $k\geq 1$, or $\ba_{k,j}=n_j-1$ for all $k\geq 1$. Therefore,
 $\pi'_j((\ba_{k})_{k\geq 1})=0$ or $1$, which implies that $F(z)\in \partial [0,1]^d$. The theorem is proved.
\end{proof}

\section{\textbf{Preliminaries on faces of $[0,1]^d$}}

We recall some notions about convex polytopes, see \cite{ZG95} or \cite{Rock}.
Let $C\subset\mathbb{R}^d$ be a convex polytope, let $F$ be a convex subset of $C$.
  We say $F$ is a \emph{face} of $C$, if any closed line segment $I\subset C$ with a relative interior point  in $F$ has both endpoints in $F$ (see \cite{Rock}).
The \emph{dimension} of a face $F$, denoted by $\dim F$, is the dimension of the smallest affine subspace containing $F$; moreover, $F$ is called an \emph{$r$-face} of $C$  if $\dim F=r$. We note that $C$ is a $d$-face of itself if $\dim C=d$. For $z\in C$, a face $F$ of $C$ is called the \emph{containing face} of $z$ if $z$ is a relative interior point of $F$.

In this section, we list some simple facts about faces of $[0,1]^d$ we need later. We call a pair $(A, B)$ an \emph{ordered partition} of $\{1,\dots,d\}$ if $A\cap B=\emptyset$ and $A\cup B=\{1,\dots,d\}$. Let $\mathbf{e}_1,\dots,\mathbf{e}_d$ be the canonical basis of $\mathbb{R}^d$. The following lemma is  obvious, see Chapter 2 of \cite{ZG95}.

\begin{lemma}\label{face}
\indent\emph{(i)} Let $(A,B)$ be an ordered partition of $\{1,\dots,d\}$  with $\#A=r$. Then the set
\begin{equation}\label{Fface}
F=\left\{\sum\limits_{j\in A}c_j\mathbf{e}_{j};\ c_j\in[0,1]\right\}+b
\end{equation}
is an $r$-face of $[0,1]^d$ if and only if $b\in T$, where
\begin{equation}\label{Ttace}
T:=\left\{\sum\limits_{j\in B}\varepsilon_j\mathbf{e}_j;\ \varepsilon_j\in\{0,1\}\right\}.
\end{equation}
\indent\emph{(ii)} For any $r$-face $F$ of $[0,1]^d$, there exists an ordered partition $(A,B)$ of $\{1,\dots,d\}$ with $\#A=r$ such that $F$ can be written as \eqref{Fface}.
\end{lemma}

We will call $F_0=\{\sum\limits_{j\in A}c_j\mathbf{e}_{j};\ c_j\in[0,1]\}$ a \emph{basic -face} related to the ordered partition $(A, B)$. If $A=\emptyset$,  we set  $F_0=\{\mathbf{0}\}$ by convention.
Let $x=\sum\limits_{j\in A}\alpha_j\mathbf{e}_{j}+\sum\limits_{i\in B}\beta_i\mathbf{e}_i\in[0,1]^d$, we define two projection maps as follows:
\begin{equation}\label{piAB}
\pi_A(x)=\sum\limits_{j\in A}\alpha_j\mathbf{e}_{j},\quad
\pi_B(x)=\sum\limits_{i\in B}\beta_i\mathbf{e}_{i}.
\end{equation}
If $F$ is an $r$-face of $[0,1]^d$, we denote by $\mathring{F}$ the relative interior of $F$.

\begin{lemma}[Huang and Rao \cite{HR20}]\label{face-to-face}
Let $F=F_0+b$ be an $r$-face of $[0,1]^d$ given by (3.1). Let $u\in\mathbb{Z}^d$. Then $\mathring{F}\cap(u+[0,1]^d)\ne\emptyset$ if and only if $u\in b-T$  where $T$ is defined in \eqref{Ttace}.
\end{lemma}

The following lemma strengths Lemma 3.1 in \cite{HR20}, and allows us to give simpler arguments in
Section 4 comparing to \cite{HR20}.

\begin{lemma}\label{dim-increase}
Let $z_0\in \partial [0,1]^d$  and  let $F$ be the containing face of $z_0$.
Let
$$g(x_1,\dots, x_d)=(a_1x_1+t_1,\dots, a_dx_d+t_d)$$
 be a map such that $a_j\in(0,1)$ and $g([0,1]^d)\subset [0,1]^d$.
Let $F'$ be the containing face of $g(z_0)$. Then

(i) $g(F)\subset F'$;

(ii) either $F'= F$ or $\dim F'\ge\dim F+1$.
\end{lemma}
\begin{proof} Notice that the assumptions imply that $t_j\in [0,1)$ for all $1\leq j\leq d$.

Let $(A,B)$ be an ordered partition in Lemma \ref{face}  which defines $F$.
By the definition of containing face, we have $z_0\in\mathring{F}$. Suppose that $g(z_0)\notin F$.

First, we prove (i).
If $F=\{z_0\}$,  the assertion holds trivially. Now we suppose that $\dim F\geq 1$.
Take any point $x\in F\setminus\{z_0\}$ and let $I$ be a closed line segment in $F$ such that $x$ is an endpoint of $I$ and $z_0$ is a relative interior point of $I$.
It is clear that $g(I)\subset g([0,1]^d)\subset[0,1]^d$. Since $g(z_0)\in F'$, we have $g(I)\subset F'$. By the arbitrariness of $x$ we obtain that $g(F)\subset F'$. Especially, we have
$\dim F'\ge\dim F.$

Next, we prove (ii).
Denote $r=\dim F$ and write $F$ as $F=F_0+b$, where $F_0$ is a basic $r$-face, and
$b\in T:=\{\sum\limits_{j\in B}\varepsilon_j\mathbf{e}_j;\ \varepsilon_j\in\{0,1\}\}$.


Suppose   that $F'$ is an $r$-face of $[0,1]^d$. Then there exists
 an ordered partition $(A',B')$ of $\{1,\dots,d\}$ such that $F'=F'_0+b'$, where $F'_0=\{\sum\limits_{j\in A'}c_j\mathbf{e}_{j};\ c_j\in[0,1]\}$ and $b'\in T'=\{\sum\limits_{j\in B'}\varepsilon_j\mathbf{e}_j;\ \varepsilon_j\in\{0,1\}\}$. Since
 \begin{equation}\label{eq1}
\diag(a_1,\dots, a_d)(F_0+b)+g(\mathbf{0})=g(F)\subset F'=F'_0+b',
\end{equation}
we have $F'_0=F_0$. Hence $A'=A$ and $B'=B$, and it follows that  $T'=T$. Applying $\pi_B$ to both sides of
\eqref{eq1}, we have
\begin{equation} \label{eq2}
b'=\pi_B(F')=\pi_B(g(F))=\pi_B(\diag(a_1,\dots, a_d)F_0)+\pi_B(g(b))=\pi_B(g(b)).
  \end{equation}
Denote $b=\sum_{j\in B} b_j\be_j$, where all $b_j\in\{0,1\}$, then
\begin{equation}\label{piB0}
b'= \sum_{j\in B} (a_jb_j+ t_j)\be_j\in T.
\end{equation}
 If $b_j=1$, then $a_jb_j+t_j>0$ and it forces $a_jb_j+t_j=1$;
   if $b_j=0$, then $a_jb_j+t_j=t_j<1$ and it forces that $a_jb_j+t_j=0$.
   Hence $b'=b$ and it follows that $F'= F$. This  confirms (ii)  and the lemma is proved.
\end{proof}

\section{\textbf{Proof of Theorem \ref{thm:inner} and  Theorem \ref{thm:subIFS}}}

Let $K=K(M, \SD)$ be a $d$-dimensional Sierpi\'nski sponge defined in Section 1.
We denote the IFS generating $K$ by $\{\Phi_i\}_{i=1}^N$.
Let $\Sigma=\{1,2,\dots,N\}$. Denote by $\Sigma^{\infty}$ and $\Sigma^{k}$ the sets of infinite words and words of length $k$ over $\Sigma$ respectively. Let $\Sigma^*=\bigcup_{k\geq0} \Sigma^{k}$ be the set of all finite words.
For $k\geq 1$, denote $\bD_k=\bD+M\bD+\dots+M^{k-1}\bD$. We call
$$K_k=M^{-k}([0,1]^d+\bD_k)$$
 the \emph{$k$-th approximation} of $K$. Clearly, $K_k\subset K_{k-1}$ and $K=\bigcap_{k=0}^\infty K_k$. For each $\sigma=\sigma_1\dots\sigma_k\in\Sigma^k$, we call $\Phi_\sigma([0,1]^d)$ a \emph{$k$-th cell}.

 For a point $z\in K$, we say $F$ is the containing face of $z$ if $F$ is a face of  $[0,1]^d$ and it is the containing face of $z$.  From now on, we always assume that
  \begin{equation}\label{assu2}
\text{$z_0$ is a trivial point of $K$ and $F$ is the containing face of $z_0$}.
\end{equation}
Let $(A,B)$ be an ordered partition in Lemma \ref{face} which defines $F$. Recall $\pi_A$ and $\pi_B$ are defined in \eqref{piAB}.
 Denote
\begin{equation}\label{layer}
\Sigma_{\sigma}=\{\omega\in\Sigma^k;\pi_A(\Phi_\omega(\mathbf{0}))
=\pi_A(\Phi_{\sigma}(\mathbf{0}))\},
\end{equation}
 \begin{equation}\label{H}
H_{\sigma}=\bigcup_{\omega\in\Sigma_\sigma}\Phi_\omega([0,1]^d).
\end{equation}
Indeed, $H_{\sigma}$ is the union of all $k$-th cells having the same projection with $\Phi_\sigma([0,1]^d)$ under $\pi_A$. The following lemma is a generalization of Lemma 3.2 in \cite{HR20}.

\begin{lemma}\label{another-z1}
Let $k>0$ and let $\sigma\in\Sigma^k$. If $H_\sigma$ is not connected or $H_\sigma\cap F=\emptyset$, then there exists $\omega^*\in\Sigma_\sigma$ such that $\Phi_{\omega^*}(z_0)\notin F$ and it is a trivial point of $K$.
\end{lemma}

\begin{proof}
Let $\dim F=r$ and let $(A,B)$ be the ordered partition  which defines $F$.
Then $F=F_0+b$, where $F_0$ is a basic $r$-face and $b=\sum_{j\in B} b_j\be_j$.

We claim that if $H_\sigma\cap F\ne\emptyset$, then there is only one $k$-th cell $\Phi_\omega([0,1]^d)$ in $H_\sigma$ which intersects $F$. Let $\Phi_\omega([0,1]^d)$ be such a cell.
 Write
 $$\Phi_\omega(z)=\diag(1/n_1^k,\dots, 1/n_d^k)z+(t_1,\dots, t_d).$$
 By Lemma \ref{dim-increase}, we have $\Phi_\omega(F)\subset F$; especially,  we have $\Phi_\omega(b)=b$, which implies
 $$\frac{b_j}{n_j^k}+t_j=b_j, \quad \text{ for } j\in B.$$
 It follows that $t_j=1-{1}/{n_j^k}$ if $b_j=1$ and $t_j=0$ if $b_j=0$.
 As for $j\in A$, the maps $\Phi_{\tilde \omega}$, ${\tilde \omega}\in H_\sigma$,  share the same $j$-th component with $\Phi_\sigma$, so $t_j$ are determined.
 Hence $\omega$ is unique in $\Sigma_\sigma$, and the claim is proved.

The  assumptions of the lemma and the above claim imply that  there is a connected component $U$ of
$H_\sigma$ such that $U\cap F=\emptyset$. Let
$W=\{\omega\in\Sigma_\sigma;\Phi_\omega([0,1]^d)\subset U\}.$

First, set
$B_0=\{j\in B;\ \text{the $j$-th coordinate of $b$ is 0}\}$,
$B_1=B\setminus B_0,$
and for each $\omega\in W$, write
$$
\pi_B(\Phi_\omega(\mathbf{0}))=\sum\limits_{j\in B_0}\alpha_j(\omega)\mathbf{e}_j+\sum\limits_{j\in B_1}\beta_j(\omega)\mathbf{e}_j.
$$
Secondly,  set
$W'=\left\{\omega\in W;\ \sum\limits_{j\in B_0}\alpha_j(\omega)\text{ attains the minimum}\right\}$, and
  take $\omega^*\in W'$  such that
$\sum\limits_{j\in B_1}\beta_j(\omega^*)$ attains the maximum in $W'$.
Then by exactly the same argument as Lemma 3.2 of \cite{HR20}, one can show that
$\Phi_{\omega^*}(z_0)$ is a trivial point of $K$. The lemma is proved.
\end{proof}

 Recall that a $d$-dimensional Sierpi\'nski sponge is degenerated if it is contained in a $(d-1)$-face   of $[0,1]^d$.

\begin{thm}\label{thm:trivial-points}
Let $K$ be a non-degenerated Sierpi\'nski self-affine sponge in $\R^d$. Let $z_0$ be a trivial point of $K$. Then there exists another trivial point $z^*\in (0,1)^d$.
\end{thm}
\begin{proof}Let  $F$ be the containing face of $z_0$ and suppose that $\dim F=r<d$. Let $(A,B)$ be the
ordered partition   which defines $F$. Let $\sigma=(\sigma_i)_{i\ge 1}\in\Sigma^\infty$ be a coding of $z_0$, and let $\sigma|_k$ denotes the prefix of $\sigma$ with length $k$. Then for each $k>0$, $z_0\in H_{\sigma|_k}\cap F$, where $H_{\sigma|_k}$ is defined in \eqref{H}.
We claim that  $K$ contains a  trivial point of the form $\Phi_\omega(z_0), \omega\in\Sigma^*$, and it is not in $F$.

Now we assume that $H_{\sigma|_k}$ is connected for all $k\geq 1$, for otherwise,
  the claim holds by Lemma \ref{another-z1}.
 Let $p$ be an integer such that $C_p$, the connected component of $K_p$ containing $z_0$, satisfies $\text{diam}(C_p)<\frac{1}{3}$. Then
 \begin{equation}\label{eq:Cp}
C_p\cap(C_p+a)=\emptyset \text{ for all } a\in\mathbb{Z}^d\setminus\{\mathbf 0\}.
\end{equation}
  That $H_{\sigma|_p}$ is connected implies  that $H_{\sigma|_p}\subset C_p$.
  Since $K$ is not degenerated, there exist $j\in\Sigma$ such that
   \begin{equation}\label{eqcapF}
   \Phi_j([0,1]^d)\cap F=\emptyset.
\end{equation}

  We consider the set $H_{j\sigma_1\dots\sigma_p}$. Let $\Sigma_j=\{i\in\Sigma;\pi_A(\Phi_i(\mathbf{0}))=\pi_A(\Phi_j(\mathbf{0}))\}$. It is easy to see that $H_{j\sigma_1\dots\sigma_p}=\bigcup_{i\in\Sigma_j}\Phi_i(H_{\sigma|_p})$.
 By \eqref{eq:Cp} and \eqref{eqcapF}, we infer that $U:=\Phi_j(H_{\sigma|_p})$ is a connected component of $H_{j\sigma_1\dots\sigma_p}$
  and $U\cap F=\emptyset$. Therefore, by Lemma \ref{another-z1}, there exists $\omega^*\in\Sigma_{j\sigma_1\dots\sigma_p}$ such that $\Phi_{\omega^*}(z_0)\notin F$ and it is a trivial point of $E$. The claim is proved.

Now by Lemma \ref{dim-increase}, the containing face of $\Phi_{\omega^*}(z_0)$  has dimension no less than $r+1$. Inductively, we can  find a trivial point $z^*$ whose containing face is $[0,1]^d$.
\end{proof}

\begin{proof}[\textbf{Proof of Theorem \ref{thm:inner}}]
  It follows directly from Theorem \ref{thm:homeo} and Theorem \ref{thm:trivial-points}.
\end{proof}

Let $\Lambda=\Lambda_{\Phi(\mathcal {D})}$ be a slicing self-affine sponge in $\R^d$.
Let $$\Lambda_k=\bigcup_{\omega\in \SD^k}\phi_\omega([0,1]^d)$$ be the $k$-th approximation of $\Lambda$.
Let $U$ be a connected component of $\Lambda_k$, following
 \cite{HR20}, we call $U$ a $k$-th \emph{island} if $U\cap\partial[0,1]^d=\emptyset.$

\begin{proof}[\textbf{Proof of Theorem \ref{thm:subIFS}}]
First, let us assume $\Lambda_{\Phi(\mathcal {D})}$ is not degenerated.

Since $\Lambda_{\Phi(\mathcal {D})}$ contains a trivial point, by Theorem \ref{thm:inner}, $\Lambda_{\Phi(\mathcal {D})}\cap (0,1)^d$ also possesses trivial points.
Consequently,  $\Lambda_k$ has a $k$-th island for $k$ large; for a careful proof of this fact, we refer to
\cite{HR20}. Without loss of generality, we assume that $\Lambda_{\Phi(\mathcal {D})}$ has a $1$-island  and we denote it by $U$. Write $U = { \cup _{\bd \in J}}{\phi _\bd} ( {{{\left[ {0,1} \right]}^d}}  )$, where $J \subset \SD$.

Let $\mathcal {D}'=\mathcal {D}\setminus J$ and let $\Lambda_{\Phi(\mathcal {D}')}$ be the slicing self-affine sponge determined by $\mathcal {D}'$.
Clearly, if a coding of a point $x$  having  infinitely many entries in $J$, then $x$ is a trivial point.
 Let $P$ be the connected part of $\Lambda_{\Phi(\mathcal {D})}$, then
$
P\subset \bigcup_{\omega\in \cup_{k\geq 0}\SD^k}\phi_\omega(\Lambda_{\Phi(\mathcal {D}')}).
$

Next, if $\Lambda_{\Phi(\mathcal {D})}$ is degenerated, then it
can be identified with  a lower dimensional slicing self-affine sponge which is non-degenerated, so   \eqref{eq:dropp} still holds.
\end{proof}

\section{\textbf{Connected indices of fractals  in Example \ref{examE}}}

Let $E$ and $E'$ be the Bedford-McMullen carpets in Example \ref{examE}. In this section we calculate the
connected indices of $E$ and $E'$.

\subsection{Graph-directed IFS}
 Let $(V,\Gamma)$ be a directed graph with a vertex set $V$ and directed-edge set $\Gamma$ where
both $V$ and $\Gamma$ are finite. The set of edges from $i$ to $j$ is denoted by $\Gamma_{i,j}$, and we assume that for any $i\in V$, there is at least one edge starting from vertex $i$. For each edge $e\in\Gamma$, there is a corresponding contraction $T_e:\mathbb{R}^n\rightarrow\mathbb{R}^n$.
 We call $(T_e)_{e\in \Gamma}$ a \emph{graph-directed IFS} (see \cite{MW88}). Its invariant sets, called \emph{graph-direct sets},
 are the unique non-empty compact sets $(K_i)_{i\in V}$ satisfying
$$
K_i=\bigcup\limits_{j\in V}\bigcup\limits_{e\in\Gamma_{i,j}}T_e(K_j),\quad  i\in V.
$$

Take any order of the vertex set $V$ and fix it.
The the $\emph{adjacency\ matrix}$ $A$ of a graph $(V,\Gamma)$ is defined  as  for any two vertices $v,w$ in $V$, $A_{v,w}$ is the number of edges in $G$ from $v$ to $w$.

\subsection{Connected part of $E$}
Let $M = \left( {\begin{array}{*{20}{c}}
   8 & 0  \\
   0 & 5  \\
\end{array}} \right)$.
 Let $\bD\subset \{0,1,\dots, 7\}\times\{0,1,\dots, 4\}$ be the digit set  illustrated in Figure \ref{K} (a). Then $E=K(M,\SD)$ is the  Bedford-McMullen carpet generated by the IFS  $\{\varphi_\bd\}_{\bd\in\SD}$
  where
  $\varphi_\bd(z)=M^{-1}(z+\bd).$

First, we determine the connected component of $K$ containing $\textbf{0}$.  Denote
\begin{equation}\label{JXY}
\begin{array}{ll}
J_{XY}=  \{(0,1),(0,2),(0,3)\},   &J_{XX}=\bD\setminus J_{XY};\\
J_{YY}= \{(7,0),(0,1),(0,2),(0,3),(7,4)\},  &J_{YX}=\bD\setminus J_{YY}.\\
\end{array}
\end{equation}
(See Figure \ref{K}.) Let $(X,Y)$ be the invariant sets of the graph-directed IFS given by Figure \ref{G}.
Then $X$ and $Y$ satisfy the set equations
$$
X=\left(\bigcup\limits_{\bd\in J_{XX}}\varphi_\bd(X)\right)\cup\left(\bigcup\limits_{\bd\in J_{XY}}\varphi_\bd(Y)\right),\quad
Y=\left(\bigcup\limits_{\bd\in J_{YX}}\varphi_\bd(X)\right)
\cup\left(\bigcup\limits_{\bd\in J_{YY}}\varphi_\bd(Y)\right).
$$

\begin{figure}[H]
\subfigure[The first iteration of $X$.]{\includegraphics[width=4.5 cm]{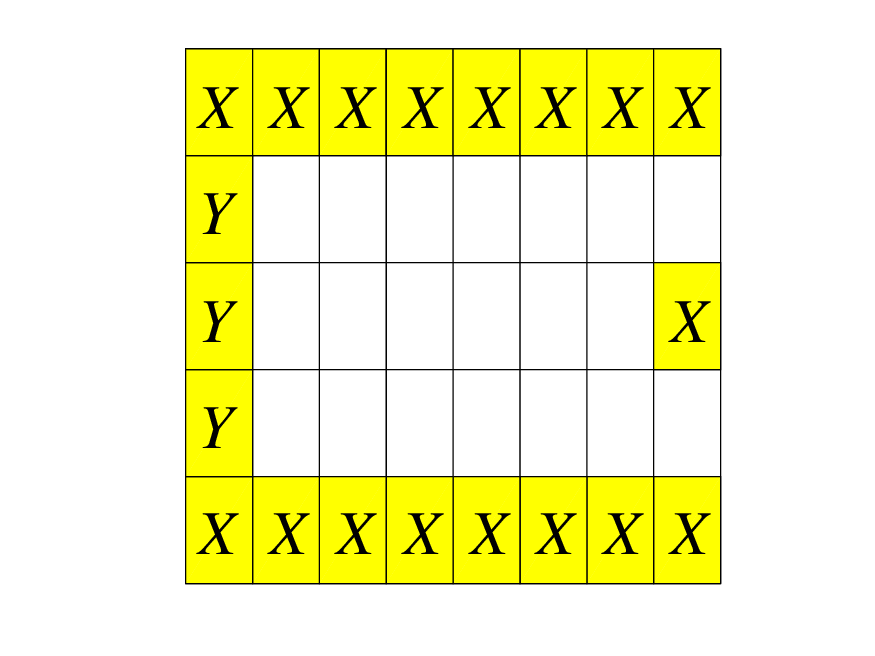}}\quad\quad
\subfigure[The first iteration of $Y$.]{\includegraphics[width=4.5 cm]{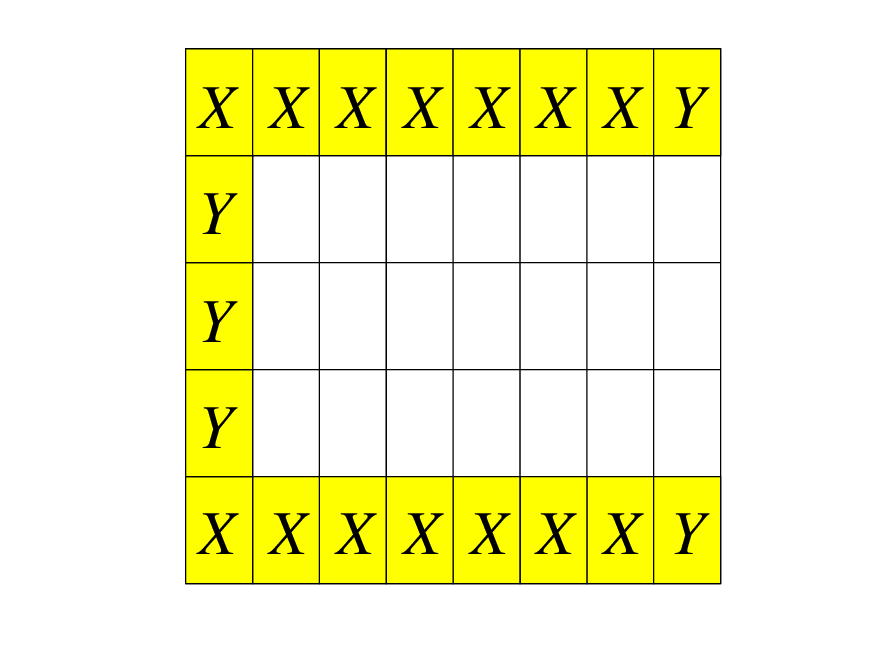}}
\caption{}
\label{K}
\end{figure}

\begin{figure}[H]
  \centering
  \includegraphics[width=6 cm]{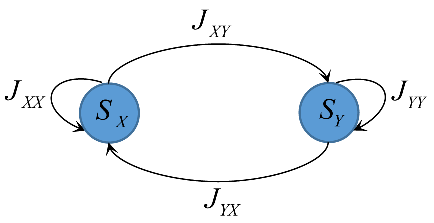}
\caption{The directed graph $(V,\Gamma)$, where $V=\{S_X, S_Y\}$.
 Each $\bd\in J_{XY}$ corresponds to an edge from $S_X$ to $S_Y$, and the corresponding  is $M^{-1}(z+\bd)$. The same holds for $J_{XX}, J_{YX}$ and $J_{YY}$.
 }
\label{G}
\end{figure}

\begin{lemma}\label{lem:Y}
 \emph{(i)} $Y$ is the connected component of $K$ containing $\mathbf{0}$;  \\
\indent\emph{(ii)} for any non-trivial connected component $C\ne Y$ of $K$, there exists $\omega\in\Sigma^*$ such that $C=\Phi_\omega(Y)$.
\end{lemma}

The prove of the above lemma is exactly the same as the proof of Lemma 5.1 in \cite{HR20}, so we omit it.

\subsection{Connectedness indices of $E$} Kenyon and Peres \cite{KP96sofic} studied the Hausdorff dimension and
box dimension of such graph-directed sets related to Bedford-McMullen carpets.
The adjacency matrix of the graph IFS \eqref{JXY}is
$$
A=\left[ {\begin{array}{cc}
   \#J_{XX} & \#J_{YX}  \\
   \#J_{XY} & \#J_{YY}  \\
\end{array}} \right]=\left[ {\begin{array}{*{20}{c}}
   {17} & {14}  \\
   3 & 5  \\
\end{array}} \right],
$$
where $\#E$ denotes the cardinality of a set $E$.
The box dimension of $X$ and $Y$ are given by
$$\dim _B X=\dim_B Y = \frac{\log \lambda}{{\log n}} + \left( {\frac{1}{{\log m}} - \frac{1}{{\log n}}} \right)\log s$$
where  $n,m$ is the expanding factors, $\lambda$ is the spectral radius of $A$, and $s$ is the number
of non-vacate rows of the digit set $\SD$.
 In our example,
$$\text{ind}_{B}(E)={\dim _B}Y = \frac{{\log \lambda }}{{\log 8}} + \left( {\frac{1}{{\log 5}} - \frac{1}{{\log 8}}} \right) \times \log 5\approx 1.662$$
where $\lambda=\frac{{22+\sqrt {312}}}{2}\approx 19.83$.

As for the Hausdorff dimension of $X$ and $Y$,    \cite{KP96sofic} proves that
$$
\dim_H X=\dim_H Y = \mathop {\lim }\limits_{k \to \infty } \frac{1}{k}{\log _m}{\sum\limits_{0 \le {i_1}, \ldots ,{i_k} \le m-1} {\left\| {{A_{{i_k}}} \cdot {A_{{i_{k - 1}}}} \cdot  \ldots  \cdot {A_{{i_1}}}} \right\|} ^{1/\sigma} }
$$
where  $A_j$ is the adjacent matrix with respect to the $j$-th row and $\sigma=\log n/\log m$.
In our example,
$${A_0} =A_4=
\left[ {\begin{array}{*{20}{c}}
   8 & 7  \\
   0 & 1  \\
\end{array}} \right], \ \
{A_1}= A_3= \left[ {\begin{array}{*{20}{c}}
   0 & 0  \\
   1 & 1  \\
\end{array}} \right],\ \
{A_2} = \left[ {\begin{array}{*{20}{c}}
   1 & 0  \\
   1 & 1  \\
\end{array}} \right].
$$
By numerical calculation, we have that $\text{ind}_H E=\dim_H Y \approx 1.61$.

\subsection{Connected indices of $E'$}
It is easy to show that the connected part of $E'$ is a Bedford-McMullen carpet indicated in Figure \ref{C}.
Hence
  $$
\text{ind}_{H}(E')=\log_5 13\approx1.54, \quad
\text{ind}_{B}(E')=\frac{{\log19+\left({1/\sigma-1}\right)\log5}}{{\log8}}\approx 1.640.
$$

\begin{figure}[H]
 \centering
  \includegraphics[width=4.6 cm]{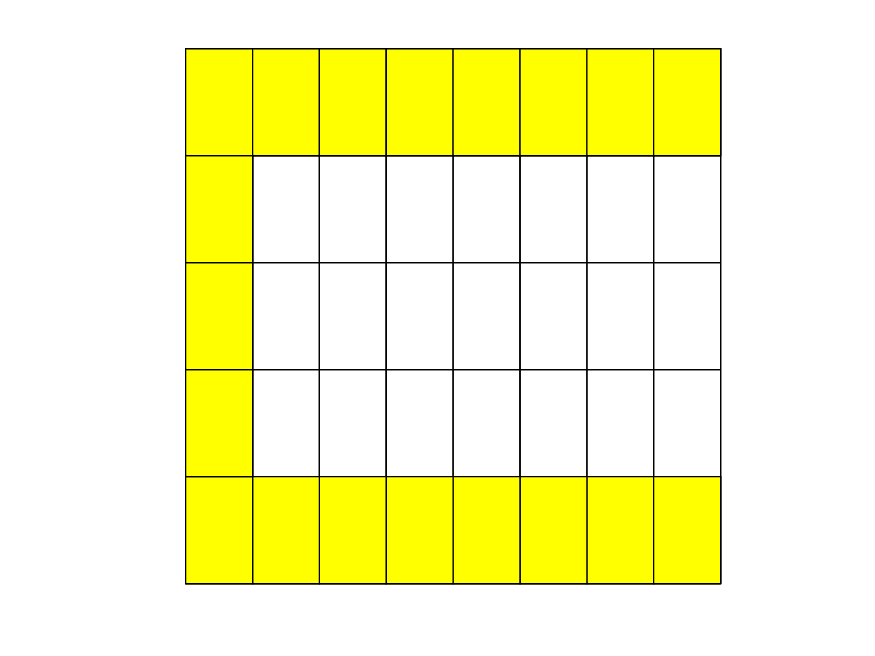}
\caption{The connected part of $E'$.}
\label{C}
\end{figure}

\section{\textbf{Trivial points of the fractal in Example \ref{exam:counter}}}
Let $E$ be the self-similar set in Example \ref{exam:counter}.
Denote $E^*=E\cup L$ where $L=\{1\}\times[1/2,1]$.
Let $H$ be the boundary of the convex hull of $E$, that is, $H$ is the trapezoid with vertices
$0, \mi, 1+\mi$ and $1+\mi/2$. Denote $\Sigma=\{1,\dots, 7\}$, see  Figure \ref{N}.
\begin{figure}[H]
 \includegraphics[width=6 cm]{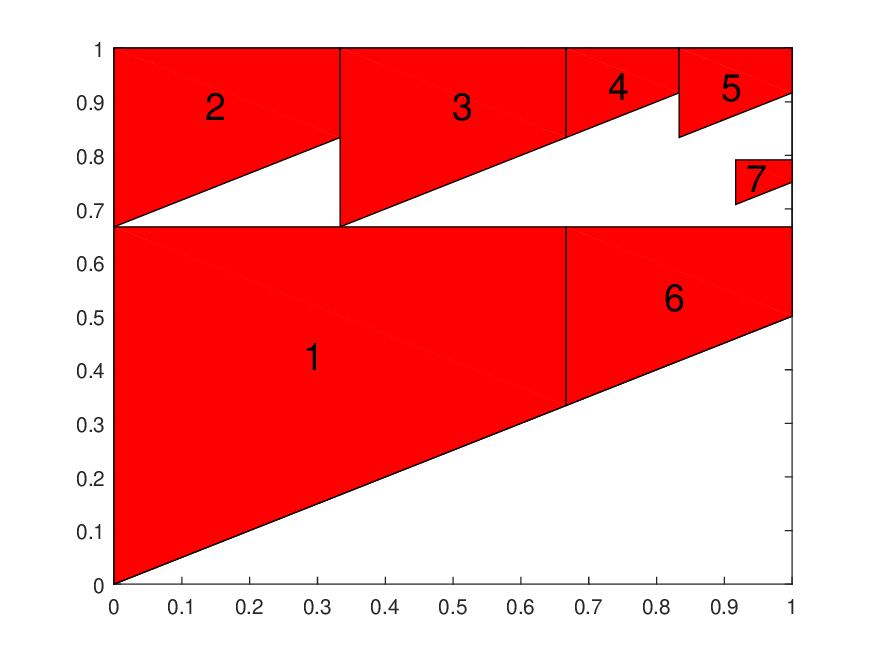}
\caption{The first iteration of $E$.}
\label{N}
\end{figure}

\begin{lemma}\label{lem:one} $E^*$ is connected.
\end{lemma}

\begin{proof} It is seen that  $H_1=\bigcup_{j=1}^7 f_j(H)$ is connected and it is a subset of  $E^*$.
Inductively, we see that for every $n\geq 1$,
$H_n=\bigcup_{\omega\in \Sigma^n}f_\omega(H)$ is connected and it is a subset
of $E^*$ . Therefore, $E^*$, as the limit of $H_n$ in Hausdorff metric, is  connected.
\end{proof}

 Let $\Sigma=A\cup B\cup C$ be a partition where
$A=\{1,2,3,4\}$, $B=\{5,6\}$ and $C=\{7\}$.
Let $X$ be the connected component of $E$ containing $0$.

\begin{lemma}\label{lem:two} Let $\omega=\omega_1\dots \omega_k\in \Sigma^k$ such that
$\omega_j\in B$ for $j<k$ and $\omega_k\in A$. Then
\begin{equation}\label{eq:XX}
f_\omega(E^*)\subset X.
\end{equation}
\end{lemma}

\begin{proof} We prove the lemma by induction on the length of $\omega$.
By Lemma \ref{lem:one}, $\bigcup_{j\in A} f_j(E^*)$ is connected, so
$\bigcup_{j\in A} f_j(E^*)\subset X$. This proves that \eqref{eq:XX}holds for $k=1$.

Let $\omega'_{k-1}=4$ if $\omega_{k-1}=5$ and $\omega'_{k-1}=1$ if $\omega_{k-1}=6$.
By induction hypothesis,
$$f_{\omega_1\dots \omega_{k-2}\omega'_{k-1}}(E^*)\subset X.$$
On one hand, $f_{\omega}(E^*)\subset f_{\omega_1\dots \omega_{k-1}}(X)$,
on the other hand, the intersection of $f_{\omega_1\dots \omega_{k-1}}(X)$
and $f_{\omega_1\dots \omega_{k-2}\omega'_{k-1}}(E^*)$ is not empty.
Therefore,  $f_{\omega_1\dots \omega_{k-1}}(X)$, and also $f_\omega(E^*)$,
are subsets of $X$. The lemma is proved.
\end{proof}

\begin{thm}\label{thm:tri} If $x$ is a trivial point of $E$, then $x\in L$.
\end{thm}

\begin{proof} Let $(\omega_k)_{k\geq 1}$ be a coding of $x$.
We are going to show that  $\omega_k\not\in A$ for all $k\geq 1$, which implies  that $x\in L$.

Suppose on the contrary that $\omega_k\in A$  and $k$ is the least integer with this property.

If $\omega_j\not\in C$ for all $1\leq j<k$, then $\omega_j\in B$ for $1\leq j<k$. By Lemma \ref{lem:two}, $x\in f_\omega(E^*)\subset X$,
which  contradicts the fact that $x$ is a trivial point.

If $\omega_j\in C$ for some $1\leq j<k$, we set $i$ to be the greatest integer such
that $\omega_i\in C$. Notice that $x$ is a trivial point of $f_{\omega_1\dots \omega_i}(E)$
since the above cylinder is disjoint from other cylinders of rank $i$.
Moreover, $f_{\omega_1\dots \omega_i}: E \to f_{\omega_1\dots \omega_i}(E)$
is a bijection, so $y=f^{-1}_{\omega_1\dots \omega_i}(x)$ is a trivial point of $E$.
A coding of $y$ is $(\omega_j)_{j\geq i+1}$. By Lemma \ref{lem:two}, $y\in f_{\omega_{i+1}\dots \omega_k}(E^*)\subset X$,
which  contradicts the fact that $y$ is a trivial point of $E$.
The theorem is proved.
\end{proof}
\bigskip


\end{document}